\documentclass[twoside,11pt,reqno]{amsart}
\usepackage{amsmath,amssymb,amscd,mathrsfs,amscd, todonotes}
\usepackage{graphics,verbatim}
\usepackage{todonotes}
\usepackage{hyperref}

\newcommand{\losemi}{{\otimes \kern -.78em \ltimes}}
\newcommand{\rosemi}{{\otimes \kern -.78em \rtimes}}

\makeatletter
\newcommand{\leqnomode}{\tagsleft@true}
\newcommand{\reqnomode}{\tagsleft@false}
\makeatother

\newtheorem{theorem}{Theorem}[subsection]

\makeatletter\let\c@fact\c@theorem\makeatother

\makeatletter\let\c@note\c@theorem\makeatother

\newtheorem{lemma}{Lemma}[subsection]
\makeatletter\let\c@lemma\c@theorem\makeatother

\makeatletter\let\c@lemma\c@theorem\makeatother

\newtheorem{quest}{Question}[subsection]
\makeatletter\let\c@alg\c@theorem\makeatother

\newtheorem{prop}{Proposition}[subsection]
\makeatletter\let\c@prop\c@theorem\makeatother

\makeatletter\let\c@conj\c@theorem\makeatother

\makeatletter\let\c@cor\c@theorem\makeatother

\newtheorem{defn}{Definition}[subsection]
\makeatletter\let\c@defn\c@theorem\makeatother

\theoremstyle{definition}

\newtheorem{remark}{Remark}[subsection]
\makeatletter\let\c@remark\c@theorem\makeatother

\makeatletter\let\c@example\c@theorem\makeatother
\numberwithin{equation}{subsection}

\usepackage[capitalise]{cleveref}
\crefname{theorem}{Theorem}{Theorems}
\crefname{fact}{Fact}{Facts}
\crefname{note}{Note}{Notes}
\crefname{lemma}{Lemma}{Lemmas}
\crefname{alg}{Algorithm}{Algorithms}
\crefname{remark}{Remark}{Remarks}
\crefname{example}{Example}{Examples}
\crefname{prop}{Proposition}{Propositions}
\crefname{conj}{Conjecture}{Conjectures}
\crefname{cor}{Corollary}{Corollaries}
\crefname{defn}{Definition}{Definitions}
\crefname{equation}{\!\!}{\!\!} 

\newcounter{listequation}

\begin{document}
\title{Varieties of $G_r$-summands in Rational $G$-modules}

\author{Paul Sobaje}
\address{Department of Mathematics \\
          University of Georgia \\
          Athens, GA 30602}
\email{sobaje@uga.edu}
\date{\today}
\subjclass[2010]{17B10 (primary), and 20G10 (secondary)} 

\begin{abstract}
Let $G$ be a simple simply connected algebraic group over an algebraically closed field $k$ of characteristic $p$, with $r$-th Frobenius kernel $G_r$.  Let $M$ be a $G_r$-module and $V$ a rational $G$-module.  We put a variety structure on the set of all $G_r$-summands of $V$ that are isomorphic to $M$, and study basic properties of these varieties.  We give a few applications of this work to the representation theory of $G$, primarily in providing some sufficient conditions for when a $G_r$-module decomposition of $V$ can be extended to a $G$-module decomposition.  In particular we are interested in connections to Donkin's tilting module conjecture, and more generally to the problem of finding a $G$-structure for the projective indecomposable $G_r$-modules.  To that end, we show that Donkin's conjecture is equivalent to determining the linearizability or non-linearizability of $G$-actions on certain affine spaces.
\end{abstract}

\maketitle

\section{Introduction}

Let $k$ be an algebraically closed field of characteristic $p>0$, $G$ a simple simply connected algebraic group over $k$, and let $G_r$ denote the $r$-th Frobenius kernel of $G$.  Every rational $G$-module restricts to a module for $G_r$.  One may therefore ask the following: given a finite dimensional $G_r$-module $M$, is it the restriction of a $G$-module?  For example, the simple $G_r$-modules all arise in this way, as first shown by Curtis in the case when $r=1$ \cite[II.3]{J}.  On the other hand, the so-called baby Verma modules for $G_r$ cannot in general be given a $G$-structure, since they usually have fundamentally different restrictions to the kernels of opposite Borel subgroups of $G$ (the exception being when the baby Verma is the $r$-th Steinberg module).  Since all Borel subgroups are conjugate in $G$, this presents an obstruction to lifting such modules to $G$.

The projective indecomposable $G_r$-modules have long been expected to have a $G$-structure \cite{HV}.  When $p \ge 2h-2$, where $h$ is the Coxeter number of $G$, they have not only been shown to have such a structure, but a unique one even, as all lifts are indecomposable tilting modules for $G$.  Donkin has conjectured that in all characteristics the projective indecomposable modules should lift to tilting modules \cite{D2}.  With the exception of $SL_3$, it is not known what happens when $p < 2h-2$.

The expectation that these modules will lift to $G$, as tilting modules or otherwise, is rooted in the fact that in all characteristics they can be realized as $G_r$-summands of $G$-modules in some nice way.  Fix a maximal torus $T$ in $G$, with $B$ the negative Borel subgroup containing $T$.  Following the notation in \cite{J}, let $Q_r(\lambda)$ be the $G_r$-projective cover of $L(\lambda)$ ($\lambda$ being $p^r$-restricted), $St_r$ the $r$-th Steinberg module, and let $\lambda^0 = (p^r - 1)\rho + w_0\lambda$.  It is known in all characteristics that $Q_r(\lambda)$ occurs as a $G_r$-summand of $St_r \otimes L(\lambda^0)$ with multiplicity one, and that its $G_r$-socle is a $G$-submodule of $St_r \otimes L(\lambda^0)$.  If we fix some such summand $Q_r(\lambda)$, the previous statements imply that for every $g \in G$, $g.Q_r(\lambda)$ is also a $G_r$-summand (in a different decomposition if $g$ moves the subspace) that projects isomorphically onto the original summand.  Thanks to a result of Pillen, Donkin's tilting module conjecture is true exactly when we can find some such $G_r$-decomposition of $St_r \otimes L(\lambda^0)$ that is $G$-stable (that is, is a $G$-decomposition).

At the other end of the spectrum, Parshall and Scott proved more recently that there is always a finite dimensional $G$-module $V$ such that $V \mid_{G_r} \cong Q_r(\lambda)^{\oplus n}$, for some $n > 0$ (see \cite{PS} for a more general result).  Based on the method that they use to show this, one does not expect that any of these $G_r$-summands are $G$-summands, however it is possible that one could be a $G$-submodule.  Even if one of these summands is not a $G$-submodule, we proved in \cite{So} that $Q_r(\lambda)$ will lift to $G$ so long as it can be lifted to $G_rB$ (this was shown only for $r=1$, but the proof will work in general).  Because of this, one need only show that there is a $B$-stable $G_r$-summand occurring in $V$ in order to get a $G$-structure on $Q_r(\lambda)$.

In each of the situations just described, $G$ permutes $G_r$-summands of some $G$-module, and the question is whether or not there is a summand that is stable as a subspace under the action of a subgroup of $G$.  A broader question one can ask is: which $G_r$-modules can even appear as $G_r$-summands in a $G$-module?  One can work out fairly easily that the baby Verma modules, for example, can not make such an appearance (see Proposition \ref{summandmeansstable} for more).

Motivated by these questions, our goal in this article is to put a $G$-variety structure on a set of $G_r$-summands, in order to lay the groundwork for the tools of geometric invariant theory to be brought to bear on the problem of finding fixed points (to be explored more fully in a later paper).  To be more specific, let $M$ be a $G_r$-summand of a $G$-module $V$.  We establish in Theorem \ref{structure} a $G$-variety structure for the set of all $G_r$-summands of $V$ that are isomorphic to $M$.  This structure is given as a homogeneous space for an algebraic group containing $G$.  We study these varieties in certain special cases that are based on the examples cited earlier.  The most complete statement about the variety structure comes when $V$ decomposes over $G_r$ into a sum of modules $M$ and $N$ such that every $G_r$-summand isomorphic to $M$ has trivial intersection with $N$.  For instance, Donkin's conjecture involves analyzing a setup such as this.  We show that the $G$-variety of all $G_r$-decompositions of $V$ into summands isomorphic to $M$ and $N$ is an affine space.  A $G$-decomposition of $V$ then occurs if and only if this affine space has a $G$-fixed point.

If $G$ acts algebraically on $\mathbb{A}^n$, then a stronger condition than having a fixed point is that the action be linearizable, meaning it is $G$-equivariantly isomorphic to a rational $G$-module.  We show in Theorem \ref{fixedislinear} that in the case described above ($V \cong M \oplus N$), $G$ acts with a fixed point on the decomposition variety if and only if the action is linearizable.  In particular, this means that the truth of Donkin's conjecture can be determined by the linearizability or non-linearizability of certain affine $G$-spaces (at least in theory).

The linearization problem for a connected reductive group acting algebraically on an affine space has been studied extensively in characteristic $0$.  In that setting, it is known that such actions need not be linearizable, however, whether there must always be a fixed point remains an open question.  On the other hand, in characteristic $p$, Kraft and Popov observe that if $G$ is not a torus, then there are many $G$-actions on affine spaces without fixed points.  Examples of such actions can be constructed from any finite dimensional $G$-module that is not semisimple (see ``Note added in proof" of \cite{KP}).  If $G$ is a torus, then any algebraic action on an affine space has a fixed point \cite{BB}.

The organization of this paper is as follows.  Relevant background information and a few technical lemmas are contained in Section 2.  In Section 3 we establish the variety structure on the set of $G_r$-summands of a $G$-module $V$.  Section 4 then gives a few applications of this work to the representation theory of $G$.  Some basic observations about algebraic actions of $G$ on $\mathbb{A}^n$ are made in Section 5, and we conclude with a brief mention about some directions to take this work in the future.

\subsection{Acknowledgments} We wish to thank Chris Drupieski and Eric Friedlander for helpful comments on an earlier version of this paper.  We thank Hanspeter Kraft for pointing us to work on reductive groups acting on affine spaces in positive characteristic.  This work was partially supported by the Research Training Grant, DMS-1344994, from the NSF.

\section{Preliminaries}

\subsection{Notation}

Throughout this paper $G$ will denote a simple simply connected group over $k$.  We fix a maximal torus $T$ contained in a Borel subgroup $B$ having unipotent radical $U$.  All notation regarding characters of $T$ and special types of highest weight modules follows that given in \cite{J}.

For any linear algebraic group $H$, we denote its Lie algebra by $\mathfrak{h}$, while its $r$-th Frobenius kernel is $H_r$.  The dual of its coordinate algebra is $\text{Dist}(H_r)$, the distribution algebra of $H_r$.  It is a finite dimensional algebra, and representations of $H_r$ correspond to modules for $\text{Dist}(H_r)$, hence we shall talk about the two interchangeably.
 
Let $V$ be a finite dimensional rational $H$-module given by a homomorphism of algebraic groups $\phi:H \rightarrow GL(V)$.  This induces an algebra map that we denote as $d\phi: \text{Dist}(H) \rightarrow \text{End}_k(V)$.  In particular, this gives a homomorphism from $\text{Dist}(H_r) \subseteq \text{Dist}(H)$ to $\text{End}_k(V)$.  The image of $H$ in $GL(V)$ acts by conjugation on $\text{End}_k(V)$, and this action stabilizes the image of $\text{Dist}(H_r)$ under $d\phi$.  We have that for all $x \in \text{Dist}(H_r)$ and $g \in H$, $d\phi(g.x) = \phi(g)d\phi(x)\phi(g)^{-1}$, where the action $g.x$ refers to the adjoint action of $g$ on $\text{Dist}(H_r)$ (Jantzen reference).

Let $M$ be an $H_r$-module. For each $h \in H$, $^hM$ denotes the module obtained by twisting $M$ by the $H_r$-automorphism $\text{Ad}(h)$ \cite[I.2.15]{J}.  We will say that $M$ is $H$-twist stable if $^hM \cong M$ for all $h \in H$ (in other contexts this might be referred to as being $H$-stable, however we must clarify our use of stability as we consider $H_r$-submodules of an $H$-module $V$, and whether or not they are stable as subspaces under the action of $H$ on $V$).

\subsection{Algebraic groups acting on varieties}

Let $H$ be a linear algebraic group over $k$.  A variety $X$ over $k$ is an $H$-variety (or $H$ acts morphically or algebraically on $X$) if there is an $H$-action on $X$ so that the map corresponding to the action,
$$\sigma: H \times X \rightarrow X,$$
is a morphism of varieties.  Given $x \in X$, we write $H.x$ for the $H$-orbit of $x$ in $X$, and $H_x$ or $C_H(x)$ for the stabilizer subgroup in $H$ of $x$.  If $Y \subseteq X$, then
$$C_H(Y) := \bigcap_{x \in Y} C_H(x), \quad N_H(Y) := \{ h \in H \mid h.Y \subseteq Y \}.$$

We recall that $C_H(x)$ and $C_H(Y)$ are closed subgroups of $H$ for any choice of $x$ and $Y$, while the subgroup $N_H(Y)$ is closed if $Y$ is closed in $X$.  Given any $h \in H$, the set of $h$-fixed points of $X$, denoted $X^h$, is closed in $X$.

The orbit $H.x$ is open in $\overline{H.x}$.  If $H$ is connected, then $H.x$ and $\overline{H.x}$ are irreducible, and $\overline{H.x}$ is the union of $H.x$ and orbits of strictly smaller dimension.  If $H$ is reductive and $X$ is affine, then there is a unique closed $H$-orbit in every orbit closure.

There is a surjective variety morphism $H \rightarrow H.x$ given by $h \mapsto h.x$.  This morphism induces a bijective morphism of $H$-varieties $$H/H_x \rightarrow H.x$$
that is an isomorphism if it is separable.  Separability occurs when when the tangent map
$$T_{1H_x}(H/H_x) \rightarrow T_x(H.x)$$
is surjective \cite[Theorem 4.3.7(ii)]{Sp}.  By \cite[Chapter 6, Theorem 3.1]{FR}, this is equivalent to checking that the tangent map
$$\mathfrak{h} = T_1(H) \rightarrow T_x(H.x)$$
is surjective.  The following lemma relates separability of orbit maps for $H$ and separability of orbit maps for a closed subgroup.

\begin{lemma}\label{separability}
Let $X$ be an $H$-variety, $x \in X$, and $F$ a closed subgroup of $H$.

\begin{enumerate}
\item If $F/F_x \cong F.x$ and $H.x = F.x$, then $H/H_x \cong H.x$.
\item If $H/H_x \cong H.x$ and $\textup{Lie}(F \cap H_x) = \textup{Lie}(F) \cap \textup{Lie}(H_x)$, then $F/F_x \cong F.x$.
\end{enumerate}

\end{lemma}

\begin{proof}
Suppose that $F.x = H.x$.  This gives an equality $T_x(F.x)=T_x(H.x)$.  Consider the orbit maps
$$\psi_H: H \rightarrow H.x, \quad \psi_F: F \rightarrow F.x$$
The tangent map $d\psi_H: \mathfrak{h} \rightarrow T_x(H.x)$ restricts to the map $d\psi_F: \mathfrak{f} \rightarrow T_x(H.x)$.  The map $d\psi_F$ is surjective by virtue of the isomorphism $F/F_x \cong F.x$, hence $d\psi_H$ must also be surjective, establishing the isomorphism $H/H_x \cong H.x$.  This proves (1).

If $H/H_x \cong H.x$, then in particular this is an $F$-equivariant isomorphism.  We may therefore consider the action of $F$ on $x \in H.x$ by its action on $H/H_x$.  But the condition that $\textup{Lie}(F \cap H_x) = \textup{Lie}(F) \cap \textup{Lie}(H_x)$ is equivalent to the separability of the morphism from $F$ to its image in $H/H_x$ \cite[Exercise 5.5.9(5b)]{Sp}.
\end{proof}

\subsection{Grassmannian Varieties}\label{Grass}

Let $V$ be an $n$-dimensional vector space over $k$.  To avoid confusion with the notation for $r$-th Frobenius kernel of $G$, the Grassmannian of $m$-dimensional subspaces of $V$ will be denoted as $\mathfrak{Grass}_m(V)$.  This set obtains a variety structure via the Pl\"{u}cker embedding which identifies it with a closed subset of $\mathbb{P}(\Lambda^m(V))$ (strictly speaking we should choose a basis for $V$ to identity it with $\mathbb{A}^n$, but we will bypass this step).  If $W$ is a subspace with basis $\{w_1,\ldots,w_m\}$, then the embedding is given by
$$W \mapsto [w_1 \wedge \ldots \wedge w_m].$$
While the map from subspaces of $V$ to $\Lambda^m(V)$ is not well-defined (that is, it is dependent on the chosen basis), the map to the projective variety is well-defined.

There is a transitive $GL(V)$-action on $\mathfrak{Grass}_m(V)$, and this action is easily seen to be compatible with the $GL(V)$-action on $\mathbb{P}(\Lambda^m(V))$ inherited from the linear $GL(V)$-action on $\Lambda^m(V)$.  Thus $\mathfrak{Grass}_m(V)$ is a homogeneous space for $GL(V)$.  Fix a subspace $W \le V$ of dimension $m$, and let $P_W \le GL(V)$ be the parabolic subgroup defined to be all $g \in G$ such that $g.W = W$.  Then $\mathfrak{p} \subseteq \mathfrak{gl}(V)$ coincides with the subalgebra of all $X$ such that $X.W \subseteq W$.  By the reasoning in \cite[Corollary 5.5.4(b)]{Sp}, it follows that the orbit map $GL(V) \rightarrow GL(V).W$ is separable (here we are identifying the right side with a $GL(V)$-orbit in $\mathbb{P}(\Lambda^m(V))$ under the Pl\"{u}cker embedding), so that we get an isomorphism
$$GL(V)/P_W \rightarrow \mathfrak{Grass}_m(V)$$
of $GL(V)$-varieties.  This then gives an alternate, but isomorphic, variety structure for $\mathfrak{Grass}_m(V)$.

One can give $\mathfrak{Grass}_m(V)$ a variety structure by local coordinates.  Fix a decomposition $V = W \oplus C$.  The set of all subspaces $W^{\prime} \subseteq V$ that project isomorphically onto $W$ under this decomposition identifies with $\text{Hom}_k(W,C)$ in the following way: if the projection map
$$\text{pr}_{W}: W^{\prime} \rightarrow W$$
is an isomorphism, then we obtain a linear map from $W$ to $C$ by
$$W \xrightarrow{(\text{pr}_{W})^{-1}} W^{\prime} \xrightarrow{\text{pr}_{C}} C.$$
Conversely, given $f \in \text{Hom}_k(W,C)$, we get a vector subspace $W_f$ defined by
$$W_f := \{ v + f(v) \mid v \in W \}.$$

Declaring sets of the form above to be open subvarieties gives $\mathfrak{Grass}_m(V)$ a covering by affine varieties.  The Pl\"{u}cker embedding sends the subvariety $\text{Hom}_k(W,C)$ to an open subvariety of $\mathbb{P}(\Lambda^m(V))$ isomorphic to $\mathbb{A}^{m(n-m)}$.

That the image of $\text{Hom}_k(W,C)$ is open in $\mathbb{P}(\Lambda^m(V))$ can be easily observed.  Let $\{c_1, \ldots, c_{n-m}\}$ be a basis for $C$, and set set $y = c_1 \wedge \ldots \wedge c_{n-m} \in \Lambda^{n-m}(V)$.  There is then a linear map
$$\Lambda^m(V) \rightarrow \Lambda^n(V) \cong k, \quad x \mapsto x \wedge y.$$
The kernel of this map defines a closed subset $Z_C \subseteq \mathbb{P}(\Lambda^m(V))$, and the set of all subspaces that project isomorphically onto $W$ are precisely those whose image under the Pl\"{u}cker embedding lands in $\mathbb{P}(\Lambda^m(V)) - Z_C$.

Finally, we recall that the Grassmannian can be given as a geometric quotient (in addition to being given as the quotient of $GL(V)$ being acted on the right by $P_W$).  Let $\text{Hom}_k(U,V)^o$ be the open subset of injective linear maps of an $m$-dimensional vector space $U$ into $V$.  Every $m$-dimensional subspace of $V$ is the image of one of these maps, and any two maps defining the same subspace differ by an automorphism of $U$.  There is a right action of the reductive group $GL(U)$ on this variety, and by the preceding statements, the orbits of this action correspond to the points in $\mathfrak{Grass}_m(V)$.  Every $GL(U)$-orbit has dimension $m^2$ and is therefore closed, and since $GL(U)$ is reductive, the geometric quotient $\text{Hom}_k(U,V)^o/GL(U)$ exists \cite[Chapter 13, Theorem 3.4]{FR}, and can be seen to be isomorphic to $\mathfrak{Grass}_m(V)$.

The following result is well known, and will be is essential in the next section.  We provide a short proof for lack of a suitable reference.

\begin{lemma}\label{sub}
Let $A$ be a $k$-algebra, and suppose that $V$ is a finite dimensional $A$-module.  Then the set of $m$-dimensional $A$-submodules of $V$ is a closed subset of $\mathfrak{Grass}_m(V)$.
\end{lemma}

\begin{proof} For each $g \in GL(V)$, the set of $m$-dimensional subspaces $Z$ such that $g.Z=Z$ is the subset $\mathfrak{Grass}_m(V)^g$, that is closed in $\mathfrak{Grass}_m(V)$.  The action of $A$ on $V$ factors through the image of $A$ in $\text{End}_k(V)$, so we may replace $A$ by this image, and therefore may assume $A$ is a subalgebra of $\text{End}_k(V)$.  By the first statement, the proof will follow from showing that $A$ can be generated as an algebra by elements in $GL(V)$.  In fact, it has a basis consisting of such elements.  For any $a \in \text{End}_k(V)$, there exists a scalar map $cI$ such that $cI+a$ is invertible (just pick $c$ so that $-c$ is not an eigenvalue of $a$).  If $a \in A$ then $cI+a \in A$.  It is now clear that one can now choose a basis for $A$ that includes $I$, and consists entirely of elements in $GL(V)$.
\end{proof}

\section{Varieties of $G_r$-summands}\label{varietysection}

\subsection{}
Throughout this section we fix $V$ to be a finite dimensional $G$-module given by the homomorphism
$$\phi:G \rightarrow GL(V).$$
Let $\mathcal{C}$ be the centralizer subgroup in $GL(V)$ of $d\phi(\text{Dist}(G_r)) \le \text{End}_k(V)$.  That is,
$$\mathcal{C} = \text{End}_{G_r}(V)^{\times} = \text{Aut}_{G_r}(V).$$
This is a connected subgroup of $GL(V)$ (it is an open subset of its Lie algebra), and is normalized by $G$ (acting via $\phi$).  Set
$$\mathcal{S} = G \ltimes \mathcal{C},$$
which is also a connected linear group.  The rational actions of $G$ and $\mathcal{C}$ extend to a rational $\mathcal{S}$-module structure on $V$.

\begin{defn}
For each $m \ge 1$, $\mathfrak{Grass}_m(V)^{G_r}$ is the closed subvariety of $\mathfrak{Grass}_m(V)$ consisting of all $G_r$-submodules of $V$ of dimension $m$, while $\mathfrak{X}^V_m$ is the open subvariety of $\mathfrak{Grass}_m(V)^{G_r}$ consisting of all $G_r$-submodules that have a $G_r$-complement in $V$ (that is, the set of all $m$-dimensional $G_r$-summands of $V$).
\end{defn}

By Lemma \ref{sub}, when $\mathfrak{Grass}_m(V)^{G_r}$ is non-empty it is a projective variety.  It is also stable under the each of the actions of $G,\mathcal{C}$, and $\mathcal{S}$, on $\mathfrak{Grass}_m(V)$.    That $\mathfrak{X}^V_m$ is an open subvariety of $\mathfrak{Grass}_m(V)^{G_r}$ follows from the local coordinates given in Section \ref{Grass}, since $\mathfrak{X}^V_m$ is the intersection of $\mathfrak{Grass}_m(V)^{G_r}$ with all open subsets of the form $\text{Hom}_k(W,C)$ for those subspaces $C$ that are $G_r$-submodules.

\begin{prop}\label{summandmeansstable}
If $M$ is a $G_r$-summand of $V$, then $M$ is $G$-twist stable.
\end{prop}

\begin{proof}
For any $g \in G$, the subspace $g.M \le V$ is also a $G_r$-summand of $V$, easily seen to be isomorphic to $^gM$.  By the Krull-Schmidt theorem, there are only finitely many possible isomorphism types for $^gM$.  For any $f \in \mathcal{C}$, $f.M$ is also a $G_r$-summand of $V$, and it is isomorphic to $M$.  We claim that $\mathcal{C}$ acts transitively on the set of all summands of a fixed isomorphism type.  Indeed, let $M^{\prime}$ be a $G_r$-summand of $V$ isomorphic to $M$, and choose $N$ and $N^{\prime}$ to be $G_r$-complements in $V$ of $M$ and $M^{\prime}$ respectively.  Choose $G_r$-module isomorphisms
$$f_M: M \rightarrow M^{\prime}, \quad f_N: N \rightarrow N^{\prime}.$$
Given $v \in V$, write $v=x+y, x \in M, y \in N$.  The linear map sending $v \mapsto f_M(x) + f_N(y)$ is an element of $\mathcal{C}$ (it is in $GL(V)$ and commutes with the image of $\text{Dist}(G_r)$ since $f_M$ and $f_N$ are $\text{Dist}(G_r)$-homomorphisms).  Thus there is an element in $\mathcal{C}$ that sends $M$ to $M^{\prime}$.

We now consider the action of $\mathcal{S}$ on the variety $X^V_m$.  There are only finitely many $\mathcal{C}$-orbits in $X^V_m$, and $G$, since it normalizes $\mathcal{C}$, permutes these orbits.  There must be at least one closed $\mathcal{C}$-orbit.  The subgroup of $G$ that stabilizes this $\mathcal{C}$-orbit is then a closed subgroup of $G$ of finite index, hence must be $G$ itself since $G$ is connected.  We can now proceed inductively by considering the action of $G$ on the complement of this closed $\mathcal{C}$-orbit in $\mathfrak{X}^V_m$ to see that $G$ must stabilize each $\mathcal{C}$-orbit.  Thus $^gM \cong M$, so that $M$ is $G$-twist stable. 

\end{proof}

\begin{defn}\label{summandset}
Let $M$ be a $G_r$-summand of $V$.  The set $\mathfrak{X}^V_M$ denotes the set of all $G_r$-summands of $V$ that are isomorphic to $M$.
\end{defn}

As shown in the proof of the last proposition, $\mathfrak{X}^V_M$ is a homogeneous space for both $\mathcal{C}$ and $\mathcal{S}$.  This then provides it with a $G$-variety structure.  Our next goal is to pin down this structure a bit further.

\begin{theorem}\label{structure}
There is an isomorphism of $\mathcal{S}$-varieties $\mathfrak{X}^V_M \cong \mathcal{S}/\mathcal{S}_M$.  In particular, this isomorphism is $G$-equivariant.
\end{theorem}

\begin{proof}
Since $\mathfrak{X}^V_M = \mathcal{S}.M$, we need to establish that the orbit map, $\mathcal{S} \rightarrow \mathcal{S}.M$, is separable.  By Lemma \ref{separability}(1), it suffices to establish the separability of the orbit map $\mathcal{C} \rightarrow \mathcal{C}.M$.  Now $\mathfrak{X}^V_M$ is an orbit for $\mathcal{C}$ acting on $\mathfrak{Grass}_m(V)$.  The orbit map $GL(V) \rightarrow GL(V).M$ is separable.  Let $P$ be the parabolic subgroup of $GL(V)$ stablizing $M$.  We have that  $\text{Lie}(\mathcal{C} \cap P) = \text{Lie}(\mathcal{C}) \cap \text{Lie}(P)$ since both $\mathcal{C}$ and $P$ are open subsets of their Lie algebras (identifying $GL(V)$ as an open subset of $\mathfrak{gl}(V)$).  The separability of the orbit map $\mathcal{C} \rightarrow \mathcal{C}.M$ then follows from Lemma \ref{separability}(2).
\end{proof}

\begin{remark}
We should point out here that the varieties of summands and submodules just introduced are different from the varieties parameterizing iso-classes of modules for a finite dimensional algebra.  See, for example, the varieties studied in \cite{HZ}.
\end{remark}

\subsection{}
More detailed information about $\mathfrak{X}^V_M$ can be obtained if we know more about the relationship between $M$ and the other $G_r$-summands of $V$.  We begin first with the case that all summands of $V$ are isomorphic.  Let
$$\text{Hom}_{G_r}(M,M^{\oplus n})^o \subseteq \text{Hom}_{G_r}(M,M^{\oplus n})$$
be the open subvariety of injective $G_r$-homomorphisms from $M$ to $M^{\oplus n}$.  This variety has a right action by $\text{Aut}_{G_r}(M)$.

\begin{prop}
Suppose that $V \cong M^{\oplus n}$ over $G_r$, and that $M$ is either simple or projective over $G_r$.  Then there is a surjective morphism
$$\textup{Hom}_{G_r}(M,V)^o \rightarrow \mathfrak{X}^V_M,$$
the fibers of which are the $\textup{Aut}_{G_r}(M)$-orbits of $\textup{Hom}_{G_r}(M,V)^o$.  If $M$ is a simple $G_r$-module, then $$\mathfrak{X}^V_M = \mathfrak{Grass}_m(V)^{G_r} \cong \mathbb{P}^{\, n-1}.$$
\end{prop}

\begin{proof}
First we note the bijection between the sets above.  In general a $G_r$-summand of $V$ isomorphic to $M$ comes from an injective $G_r$-module map $M$ to $V$.  If $M$ is simple or projective over $G_r$ then the image of every such map is a $G_r$-summand.  Any two injections defining the same summand differ by a $G_r$-automorphism of $M$.  This establishes that the morphism from $\textup{Hom}_k(M,V)^o$ to $\mathfrak{Grass}_m(V)$ in Section 2.3 restricts to one from $\textup{Hom}_{G_r}(M,V)^o$ to $\mathfrak{X}^V_M$, and the fibers are then clearly the $\textup{Aut}_{G_r}(M)$-orbits.

When $M$ is simple, then every $G_r$-submodule of $V$ of dimension $m$ is isomorphic to $M$, so that $\mathfrak{X}^V_M = \mathfrak{Grass}_m(V)^{G_r}$.. We also have
$$\text{Aut}_{G_r}(M) \cong \mathbb{G}_m,$$
and
$$\text{Hom}_{G_r}(M,V) \cong k^{\oplus n},$$
so that
$$\mathfrak{X}^V_M \cong (k^{\oplus n} - \{0\})/\mathbb{G}_m.$$
\end{proof}

\subsection{}

Let $N$ be a $G_r$-complement to the summand $M$ of $V$, and let $\text{pr}_M$ and $\text{pr}_N$ denote the projection maps of $V$ onto $M$ and $N$ respectively.  These are $G_r$-module maps.  We note that for any subspace $W \le V$, the projection map $\text{pr}_M \mid_W$ is injective if and only if $W \cap N = \{0\}$.

We are now interested in the ``opposite" situation from that considered in the previous subsection.  That is, the case where every $G_r$-summand of $V$ that is isomorphic to $M$ is a complement to $N$ (this evidently fails to hold if $N$ contains a summand that is isomorphic to a summand of $M$).  In this case, it can quickly be worked out that there is a bijection between $\mathfrak{X}^V_M$ and $\text{Hom}_{G_r}(M,N)$.  Much of the work that now follows serves to make sure that this is true as $G$-varieties, since there is no obvious $G$-variety structure on $\text{Hom}_{G_r}(M,N)$ if one of these modules is not a $G$-module (or, not yet \textit{known} to be one at any rate).

We begin with a proposition that explores various ways to characterize this case of interest.

\begin{prop}\label{differentways}
Let $V$ be a $G$-module, with $V \cong M \oplus N$ over $G_r$.  The following conditions are equivalent:

\begin{enumerate}
\item Every $G_r$-summand isomorphic to $M$ is a $G_r$-complement to $N$.
\item Every $G_r$-summand isomorphic to $N$ is a $G_r$-complement to $M$.
\item Every $G_r$-summand isomorphic to $M$ is a $G_r$-complement to every $G_r$-summand isomorphic to $N$.
\item $\textup{Aut}_{G_r}(V) \cong \textup{Aut}_{G_r}(M) \times \textup{Hom}_{G_r}(M,N) \times \textup{Hom}_{G_r}(N,M) \times \textup{Aut}_{G_r}(N)$.
\item Every $X \in \textup{Hom}_{G_r}(M,N) \oplus \textup{Hom}_{G_r}(N,M)$ is a nilpotent endomorphism of $V$.
\end{enumerate}

\end{prop}

\begin{proof}
$(1) \Rightarrow (2)$: As shown in the proof of Proposition \ref{summandmeansstable}, for every $f \in \text{Aut}_{G_r}(V)$ the subspace $f.M$ is a $G_r$-summand of $V$ that is isomorphic to $M$, and every such summand is of this form.  The same holds for summands isomorphic to $N$.  If each summand isomorphic to $M$ is a $G_r$-complement to $N$, then $f.M \cap N = \{0\}$ for every $f \in \text{Aut}_{G_r}(V)$.  This immediately implies that we also have $M \cap f^{-1}.N = \{0\}$ for every $f \in \text{Aut}_{G_r}(V)$, which implies (2).

$(2) \Rightarrow (3)$: Suppose that (3) does not hold.  Then there are $f_1, f_2 \in \text{Aut}_{G_r}(V)$ such that $f_1.M \cap f_2.N \ne \{0\}$.  This implies that $M \cap ({f_1}^{-1}f_2).N \ne \{0\}$, which cannot happen when (2) holds.

$(3) \Rightarrow (4)$: $\text{Aut}_{G_r}(V)$ is a subset of $\text{End}_{G_r}(V)$, and the latter has a vector space decomposition into
$$\textup{End}_{G_r}(M) \oplus \textup{Hom}_{G_r}(M,N) \oplus \textup{Hom}_{G_r}(N,M) \oplus \textup{End}_{G_r}(N).$$
Viewing each as affine varieties, we can write this decomposition as
$$\text{End}_{G_r}(V) \cong \textup{End}_{G_r}(M) \times \textup{Hom}_{G_r}(M,N) \times \textup{Hom}_{G_r}(N,M) \times \textup{End}_{G_r}(N).$$
If $f \in Aut_{G_r}(V)$, then since $f.M$ is a complement to $N$, and $f.N$ is a complement to $M$, we have isomorphisms
$$\text{pr}_M: f.M \xrightarrow{\sim} M, \quad \text{pr}_N: f.N \xrightarrow{\sim} N.$$
Therefore, as an element in $\text{End}_{G_r}(V)$, $f$ has $\text{End}_{G_r}(M)$-component contained in $\text{Aut}_{G_r}(M)$, and likewise for the $\text{End}_{G_r}(N)$-component.  This proves the inclusion
$$\text{Aut}_{G_r}(V) \subseteq \text{Aut}_{G_r}(M) \times \text{Hom}_{G_r}(M,N) \times \text{Hom}_{G_r}(N,M) \times \text{Aut}_{G_r}(N).$$
Conversely, suppose that $f$ is in the set on the right.  This means that $\text{pr}_M(f.M)=M$ and $\text{pr}_N(f.N)=N$, so that $f.M$ and $f.N$ are summands isomorphic to $M$ and $N$ respectively.  They are therefore $G_r$-complements by (3), and so $f.V=V$, hence $f \in \text{Aut}_{G_r}(V)$, proving that the inclusion above is an equality.

$(4) \Rightarrow (5)$: Let $X \in \text{Hom}_{G_r}(M,N) \oplus \text{Hom}_{G_r}(N,M)$.  Then $X$ is nilpotent if and only if $0$ is its only eigenvalue.  However, the description of $\text{Aut}_{G_r}(V)$ in (4) implies that any non-zero scalar multiple of the identity map on $V$, when added to $X$, is still invertible.  If $X$ had an eigenvalue $c \in k^{\times}$, then adding $-c$ times the identity map to $X$ would give an endomorphism with non-trivial kernel, hence not be invertible.  Therefore $0$ is the only eigenvalue of $X$.

$(5) \Rightarrow (4)$:
Suppose now that every $X \in \text{Hom}_{G_r}(M,N) \oplus \text{Hom}_{G_r}(N,M)$ is nilpotent.  Write the two components of $X$ as
$$X = X_{M,N} + X_{N,M}.$$
Now let $A \in \text{Aut}_{G_r}(M), B \in \text{Aut}_{G_r}(N)$, and write $I_V,I_M,I_N$ for the respective identity maps on $V,M$ and $N$.  We can then factor $A + X + B$ as
$$A + X_{M,N} + X_{N,M} + B = (A + I_N)(I_V + X_{M,N} + A^{-1}X_{N,M}B^{-1})(I_M + B).$$
By assumption, $X_{M,N} + A^{-1}X_{N,M}B^{-1}$ is a nilpotent element.  Therefore, the factorization above gives us that
$$\det(A + X_{M,N} + X_{N,M} + B)=\det(A)\det(B)\ne 0.$$
Thus we have shown that
$$\text{Aut}_{G_r}(M) \times \text{Hom}_{G_r}(M,N) \times \text{Hom}_{G_r}(N,M) \times \text{Aut}_{G_r}(N) \subseteq \text{Aut}_{G_r}(V).$$

Write $\mathcal{A} = \text{Aut}_{G_r}(M) \times \text{Hom}_{G_r}(M,N) \times \text{Hom}_{G_r}(N,M) \times \text{Aut}_{G_r}(N)$.  It is easy to see that $\mathcal{A}$ and $\text{Aut}_{G_r}(V)$  have the same dimension, since they are both open subsets of $\text{End}_{G_r}(V)$.  This also implies that $\mathcal{A}$ is an open subvariety of $\text{Aut}_{G_r}(V)$.  Since $\text{Aut}_{G_r}(V)$ is connected, any open subgroup must be the entire group (otherwise the cosets would give a decomposition into disjoint open sets).  We therefore must show that $\mathcal{A}$ is a subgroup.  We can represent multiplication $\text{End}_{G_r}(V) \subset \text{End}_k(V)$ in terms of block matrices, with a decomposition as above.  We have
$$\left[ \begin{array}{c|c} 
A & X_{N,M}\\
\hline
X_{M,N} & B\\
\end{array}\right]
\left[ \begin{array}{c|c} 
A^{\prime} & X_{N,M}^{\prime}\\
\hline
X_{M,N}^{\prime} & B^{\prime}\\
\end{array}\right] = \left[ \begin{array}{c|c} 
AA^{\prime} + X_{N,M}X_{M,N}^{\prime}  & *\\
\hline
* & X_{M,N}X_{N,M}^{\prime} + BB^{\prime}\\
\end{array}\right].$$
As an element in $ \text{End}_{G_r}(M)$, we can write
$$AA^{\prime} + X_{N,M}X_{M,N}^{\prime} =  AA^{\prime}\left(I_M + (AA^{\prime})^{-1}X_{N,M}X_{M,N}^{\prime}\right).$$
Since $X_{M,N} + (AA^{\prime})^{-1}X_{N,M}$ is assumed to be nilpotent in $\text{End}_{G_r}(V)$, it must be the case that $(AA^{\prime})^{-1}X_{N,M}X_{M,N}^{\prime}$ is nilpotent in $\text{End}_{G_r}(M)$.  Therefore $AA^{\prime} + X_{N,M}X_{M,N}^{\prime}$ is in $\text{Aut}_{G_r}(M)$.  Similarly, $X_{M,N}X_{N,M}^{\prime} + BB^{\prime} \in \text{Aut}_{G_r}(N)$.  This shows that $\mathcal{A}$ is closed under multiplication, and one can apply a related factorization argument to show that $\mathcal{A}$ is closed under taking inverses, so is a subgroup of $\text{Aut}_{G_r}(V)$.

$(4) \Rightarrow (1)$: Since every summand isomorphic to $M$ is of the form $f.M$, and the component of $f$ in $\text{End}_{G_r}(M)$ is in $\text{Aut}_{G_r}(M)$, it follows that $\text{pr}_M(f.M) = M$, so that $f.M$ is a complement to $N$.
\end{proof}

\begin{theorem}\label{maincase}
Let $V \cong M \oplus N$ over $G_r$, and suppose that any of the equivalent conditions in Proposition \ref{differentways} hold.  Then 
$$\mathfrak{X}^V_M \cong \textup{Hom}_{G_r}(M,N), \qquad \mathfrak{X}^V_N \cong \textup{Hom}_{G_r}(N,M).$$
Furthermore, the set of all $G_r$-decompositions of $V$ into modules isomorphic to $M$ and $N$ can be identified with the variety $\mathfrak{X}^V_M \times \mathfrak{X}^V_N$.
\end{theorem}

\begin{proof}
From the proposition, the multiplicative structure in $\mathcal{C}$ yields a variety isomorphism
$$\mathcal{C} \cong \text{Aut}_{G_r}(M) \times \text{Hom}_{G_r}(M,N) \times \text{Hom}_{G_r}(N,M) \times \text{Aut}_{G_r}(N),$$
while the subgroup that stabilizes the summand $M$ is the subgroup (described via the variety decomposition)
$$\text{Aut}_{G_r}(M) \times \text{Hom}_{G_r}(N,M) \times \text{Aut}_{G_r}(N).$$
As in the previous proof, represent an element in $\mathcal{C}$ as a block matrix of the form
$$\left[ \begin{array}{c|c} 
A & X_{N,M}\\
\hline
X_{M,N} & B\\
\end{array}\right].$$
This element factors as
$$\left[ \begin{array}{c|c} 
I_M & 0\\
\hline
X_{M,N}A^{-1} & I_N\\
\end{array}\right]
\left[ \begin{array}{c|c} 
A & X_{N,M}\\
\hline
0 & -X_{M,N}A^{-1}X_{N,M} + B\\
\end{array}\right],$$
with $-X_{M,N}A^{-1}X_{N,M} + B \in \text{Aut}_{G_r}(N)$ by the same arguments applied in the previous proof.  Thus
$$\mathcal{C}/\mathcal{C}_M \cong \text{Hom}_{G_r}(M,N).$$
Swapping the roles of $M$ and $N$ proves the same for $\mathfrak{X}^V_N$.

Finally, when looking at the set of decompositions, we note that $\mathcal{C}$ acts transitively on this set, and that the stabilizer of the fixed decomposition $V \cong M \oplus N$ is the subgroup $\text{Aut}_{G_r}(M) \times \text{Aut}_{G_r}(N)$.  Thus we have that
$$\mathfrak{X}^V_M \times \mathfrak{X}^V_N \cong \mathcal{C}/(\text{Aut}_{G_r}(M) \times \text{Aut}_{G_r}(N)) \cong \text{Hom}_{G_r}(M,N) \times \text{Hom}_{G_r}(N,M).$$
\end{proof}

\section{Applications to Representation Theory}

Our aim in introducing the varieties found in Section \ref{varietysection} is to develop another tool to use in studying the representation theories of $G$, $G_r$, and various other subgroup schemes of $G$ (such as $G_rT$ and $G_rB$).  In this section we give various representation-theoretic results, many of which rely in an essential way on the work found in the previous section.

\subsection{}

In this subsection, we will work with the setup in Theorem \ref{maincase}.  That is, let $V \cong M \oplus N$ over $G_r$, and suppose that the conditions in Proposition \ref{differentways} hold.

\begin{prop}\label{Bisenough}
If there is a $G_r$-summand of $V$ isomorphic to $M$ that is stable under the action of $B$, then it is stable under $G$.
\end{prop}

\begin{proof}
Since the variety $\mathfrak{X}^V_M$ is an affine space, a fixed point for $B$ gives a morphism from the projective variety $G/B$ to an affine variety.  But this must be a constant morphism, so this point is a fixed point for all of $G$.
\end{proof}

\begin{remark}
We note that an arbitrary $G$-module (not satisfying the hypotheses in this section) can have a $G_rB$-submodule that is a $G_r$-summand and yet is not a $G$-submodule.  An easy example is to take the $G$-module $L(\lambda) \otimes L(\mu)^{(r)}$.  Let $v\ne0$ be a lowest-weight vector of $L(\mu)^{(r)}$.  Then $L(\lambda) \otimes v$ spans a $G_rB$-submodule and $G_r$-summand of $L(\lambda) \otimes L(\mu)^{(r)}$.  This is, of course, not a $G$-submodule since $L(\lambda) \otimes L(\mu)^{(r)}$ is a simple $G$-module.
\end{remark}

We recall again that the action of $G$ on $\mathbb{A}^n$ is said to be linearizable if $\mathbb{A}^n$ is $G$-equivariantly isomorphic to a rational $G$-module.

\begin{theorem}\label{fixedislinear}
Let $H \le G$.  The following are equivalent:
\begin{enumerate}

\item There is a $G_rH$-decomposition $V \cong M \oplus N$. 
\item $H$ acts with a fixed point on $\mathfrak{X}^V_M \times \mathfrak{X}^V_N$.
\item The action of $H$ on $\mathfrak{X}^V_M \times \mathfrak{X}^V_N$ is linearizable.

\end{enumerate}
\end{theorem}

\begin{proof}
It is clear by the definition of $\mathfrak{X}^V_M \times \mathfrak{X}^V_N$ that $(1)$ and $(2)$ are equivalent.  Also, $(3) \Rightarrow (2)$.  So we will show that $(1) \Rightarrow (3)$.  Suppose there is such a decomposition.  Returning to the proof of Theorem \ref{maincase}, we identified $\mathfrak{X}^V_M$ with the affine space $\text{Hom}_{G_r}(M,N)$.  The decomposition of $V$ over $H$ is equivalent to saying that the image of $H$ in $GL(V)$ is contained in the subgroup $GL(M) \times GL(N)$.  Thus, conjugation by the image of $H$ on $\mathcal{C}$ preserves the variety decomposition
$$\mathcal{C} \cong \text{Aut}_{G_r}(M) \times \text{Hom}_{G_r}(M,N) \times \text{Hom}_{G_r}(N,M) \times \text{Aut}_{G_r}(N).$$
Now, this decomposition includes in a variety decomposition of $\mathcal{S}$ since $\mathcal{S} \cong G \times \mathcal{C}$ as a variety.  From this decomposition we get a closed embedding of varieties
$$\iota: \text{Hom}_{G_r}(M,N) \times \text{Hom}_{G_r}(N,M) \rightarrow \mathcal{S}$$
defined in the obvious way.

Let $x \in \mathfrak{X}^V_M \times \mathfrak{X}^V_N$ be the element corresponding to the fixed decomposition in (1).  Then $H$ is in the stabilizer subgroup $\mathcal{S}_x$.  Thus, the action of $h \in H$ on the coset $g\mathcal{S}_x$ is
$$h.g\mathcal{S}_x = hgh^{-1}h\mathcal{S}_x=hgh^{-1}\mathcal{S}_x.$$
Since $H$ normalizes the subset $\iota(\text{Hom}_{G_r}(M,N) \times \text{Hom}_{G_r}(N,M)) \subseteq \mathcal{S}$, which also serves as a complete set of coset representatives for $\mathcal{S}_x$, it follows that the action of $H$ on
$$\mathfrak{X}^V_M \times \mathfrak{X}^V_N \cong \text{Hom}_{G_r}(M,N) \times \text{Hom}_{G_r}(N,M)$$
is via conjugation on the variety on the right, which just corresponds to the standard linear action on the $G_r$-fixed points of $\text{Hom}_k(M,N) \oplus \text{Hom}_k(N,M)$.
\end{proof}

\begin{prop}\label{Udecomposition}
If there is a direct sum decomposition $V \cong M \oplus N$ over $G_rU$, then some such decomposition extends to a decomposition over $G$.
\end{prop}

\begin{proof}
We have that $\mathfrak{X}^V_M \times \mathfrak{X}^V_N$ is an affine space.  If there is a direct sum decomposition for $G_rU$, then the fixed point space $(\mathfrak{X}^V_M \times \mathfrak{X}^V_N)^U$ is also an affine space by Theorem \ref{fixedislinear}, and is stable under the action of $T$ since $T$-normalizes $U$.  By \cite{BB}, there must be a $T$-fixed point on $(\mathfrak{X}^V_M \times \mathfrak{X}^V_N)^U$, hence a $B$-fixed point.  But any $B$-fixed point is a $G$-fixed point.
\end{proof}

\begin{prop}
A closed $G$-orbit on $\mathfrak{X}^V_M \times \mathfrak{X}^V_N$ corresponds to a reductive subgroup $H \le G$, and a $G_r$-decomposition $V \cong M \oplus N$ that extends to a decomposition over $G_rH$, but does not extend to one over $G_rH^{\prime}$ for any $H < H^{\prime} \le G$.
\end{prop}

\begin{proof}
Since $\mathfrak{X}^V_M \times \mathfrak{X}^V_N$ is affine, a closed orbit corresponds to a stabilizer subgroup that is reductive.  The result then follows immediately.
\end{proof}

\subsection{}

We want to improve on the results in the previous subsection by comparing the actions of $G$ on $\mathfrak{X}^V_M$ and $\mathfrak{X}^V_N$.  In particular, it would be useful to know when a $G$-fixed point in one of these varieties implies the existence of a fixed point in the other.  In such cases, one would need only find a $G$-stable summand isomorphic to $M$ or one isomorphic to $N$ in order to obtain a $G$-decomposition $V \cong M \oplus N$.

We begin by observing that the dimensions of $\mathfrak{X}^V_M$ and $\mathfrak{X}^V_N$ are the same whenever $M$ or $N$ is a projective $G_r$-module.

\begin{prop}
Let $V$ be a $G$-module, and $V \cong M \oplus N$ over $G_r$.  If either $M$ or $N$ is projective over $G_r$, then $\dim(\mathfrak{X}^V_M) = \dim(\mathfrak{X}^V_N)$.
\end{prop}

\begin{proof}
Let $Q$ be any finite dimensional projective $G_r$-module.  It decomposes into a direct sum of modules of the form $Q_r(\mu)$, $\mu \in X_r(T)$.  As $Q_r(\mu)$ is both the projective cover and injective hull of $L(\mu)$ over $G_r$, it follows that $Q_r(0)^* \cong Q_r(0)$, and $Q_r(\mu)^* \not\cong Q_r(0)$ if $\mu \neq 0$.  The dimension of $Q^{G_r}$ equals the number of indecomposable summands of $Q$ that are isomorphic to $Q_r(0)$.  This is then the same as the number of summands of $Q^*$ isomorphic to $Q_r(0)$.  Thus we have
$$\text{dim}_k\left(Q^{G_r}\right) = \text{dim}_k\left((Q^*)^{G_r}\right).$$
If either $M$ or $N$ is projective, we have that the $G_r$-module $\text{Hom}_k(M,N) \cong M^* \otimes N$ is projective.  We then have
$$\text{Hom}_k(M,N)^{G_r} \cong (M^* \otimes N)^{G_r}$$
and
$$\text{Hom}_k(N,M)^{G_r} \cong (N^* \otimes M)^{G_r} \cong ((M^* \otimes N)^*)^{G_r}.$$
The dimensions of $(M^* \otimes N)^{G_r}$ and $((M^* \otimes N)^*)^{G_r}$ are the same, establishing the equality $\dim(\mathfrak{X}^V_M) = \dim(\mathfrak{X}^V_N)$.
\end{proof}

We recall the antiautomorphism $\tau$ of $G$ that swaps each positive root subgroup with its negative counterpart, and is the identity map on $T$ (this generalizes to other reductive groups the transpose map for $GL_n$).  One then gets a $\tau$-twist of a finite dimensional $G$-module $V$ by letting $^{\tau}V = V^*$ as a vector space, and by letting $g.\varphi(v) = \varphi(\tau(g).v)$.  Twisting by $\tau$ gives a contravariant functor from the category of finite dimensional $G$-modules to itself (cf. \cite[II.2.12]{J}).

By general properties (for example, in the proof of \cite[Lemma I.4.4]{J}), for finite dimensional $G$-modules $V$ and $W$, there is a canonical isomorphism $(V \otimes W)^* \cong V^* \otimes W^*$.  Since $\tau \times \tau$ defines an antiautomorphism of $G \times G$, and the action of $G$ on $V \otimes W$ and $V^* \otimes W^*$ comes via the action of $G \times G$ on these modules ($G$ injecting to this group via the diagonal map), it follows that $^{\tau}(M \otimes N) \cong {^{\tau}M} \otimes {^{\tau}N}$.

As $\tau$ restricts to an antiautomorphism of $G_r$, all of this also holds for modules over $G_r$ and $G_rT$.

\begin{prop}
Let $V$ be a $G$-module, and $V \cong M \oplus N$ over $G_r$, and suppose that the conditions in Proposition \ref{differentways} hold.
\begin{enumerate}
\item There is a $G$-equivariant isomorphism $\mathfrak{X}^{V}_{M} \cong \mathfrak{X}^{V^*}_{N^*}$.
\item There is a variety isomorphism $\psi: \mathfrak{X}^{V}_{M} \xrightarrow{\sim} \mathfrak{X}^{^{\tau}V}_{^{\tau}N}$, such that $\tau(g^{-1}).\psi(x) = \psi(g.x)$ for all $g \in G$ and $x \in \mathfrak{X}^{V}_{M}$.
\end{enumerate}

\end{prop}

\begin{proof}
(1) The $G_r$-decomposition $V \cong M \oplus N$ gives a $G_r$-decomposition $V^* \cong M^* \oplus N^*$, where
$$N^* = \{f \in V^* \mid f(M) = 0\}.$$
For any $g \in G$, $g.M$ is also a $G_r$-complement to $N$ (possibly the same if $g$ does not move $M$).  This yields a summand $(N^*)^{\prime}$ of $V^*$, defined by
$$(N^*)^{\prime} = \{f \in V^* \mid f(g.M) = 0\}.$$
On the other hand, the action by $g$ on $V^*$ sends $N^*$ to the summand
$$g.N^* = \{g.f \in V^* \mid f(M) = 0\} = \{f \in V^* \mid f(g.M) = 0\}.$$
In this way we get a $G$-equivariant isomorphism from $\mathfrak{X}^V_M$ to $\mathfrak{X}^{V^*}_{N^*}$ by sending a $G_r$-summand isomorphic to $M$ to the linear functional which is $0$ on this subspace.

(2) Since the vector space ${^{\tau}V}=V^*$, many of the steps above hold verbatim here also.  We can match up $G_r$-summands of $V$ isomorphic to $M$ with $G_r$-summands of ${^{\tau}V}$ isomorphic to ${^{\tau}N}$ using the same map (because of the vector space isomorphism).  The key difference is that the action of $G$ on ${^{\tau}V}$ involves $\tau(g)$ rather than $g^{-1}$ (which is the usual action of $G$ on the linear dual of a $G$-module).  Therefore
$$\tau(g^{-1}).{^{\tau}N} = \{\tau(g^{-1}).f \in V^* \mid f(M) = 0\} = \{f \in V^* \mid f(g.M) = 0\}.$$

\end{proof}

\begin{theorem}\label{Utausplit}
Let $V,M$ and $N$ be as in the previous proposition.  Suppose further that ${^{\tau}V}\cong V$ over $G$, and ${^{\tau}M} \cong M$ over $G_r$.  If either $\mathfrak{X}^V_M$ or $\mathfrak{X}^V_N$ has a fixed point for $U$, then both do, and there is a $G$-decomposition $V \cong M \oplus N$.
\end{theorem}

\begin{proof}
The assumptions on $M$ and $V$ imply that ${^{\tau}N} \cong N$ over $G_r$.  If there is a fixed point for $U$ in $\mathfrak{X}^V_M$, then by the previous proposition there is a fixed point for $\tau(U)=U^+$ acting on $\mathfrak{X}^V_N=\mathfrak{X}^{^{\tau}V}_{^{\tau}N}$.  Let $x$ be such a fixed point.  Then $w_0.x$ is a fixed point for $U=w_0(U^+){w_0}^{-1}$.  Thus $U$ fixes some summand isomorphic to $N$, giving a $G_rU$-decomposition of $V$, and hence we get some $G$-decomposition by Proposition \ref{Udecomposition}.
\end{proof}

\begin{remark}
If $V \cong V^*$ and $M \cong M^*$, then the same result holds with an even shorter proof.
\end{remark}

\subsection{}

Let $\lambda \in X_r(T)$, and let $D_r(\lambda)$ denote the $G$-module $L((p^r-1)\rho + w_0\lambda) \otimes St_r$.  There is a $G_r$-module decomposition
$$D_r(\lambda) \cong Q_r(\lambda) \oplus C,$$
where $C$ is a projective $G_r$-module, and does not have any $G_r$-summand isomorphic to $Q_r(\lambda)$ \cite[II.11.9(3)]{J}.  In fact, there is a $G_rT$-decomposition
\begin{equation}\label{Donkindecomp}D_r(\lambda) \cong \widehat{Q}_r(\lambda) \oplus \sum_{\mu \in X_r(T), \mu \ne \lambda} \left(\widehat{Q}_r(\mu) \otimes \text{Hom}_{G_r}(L(\mu),D_r(\lambda))\right).\end{equation}
This can be seen by decomposing the $G_r$-socle of $D_r(\lambda)$ as a $G$-module, which then describes the $G_rT$-socle (it is the same as the $G_r$-socle, but also distinguishes the $T$-weights), which in turn determines the $G_rT$-summands.  Donkin's tilting module conjecture says that the indecomposable tilting module of highest weight $(p^r-1)\rho + w_o\lambda$ is isomorphic to $Q_r(\lambda)$ over $G_r$.  This is equivalent to some decomposition as in (\ref{Donkindecomp}) being a $G$-module decomposition \cite[E.9]{J}.

There is a $G$-module isomorphism ${^{\tau}D_r(\lambda)} \cong D_r(\lambda)$, and $G_r$-isomorphisms ${^{\tau}Q_r(\lambda)} \cong Q_r(\lambda)$, ${^{\tau}C} \cong C$.  This is then an example of a $G$-module with $G_r$-summands that satisfy the conditions in Theorem \ref{Utausplit}.  Moreover, both $\text{soc}_{G_r} \, Q_r(\lambda)$ and $\text{soc}_{G_r} \, C$ are canonically $G$-submodules of $D_r(\lambda)$, and dually, $Q_r(\lambda) \oplus \text{rad}_{G_r} \, C$ and $\text{rad}_{G_r} \, Q_r(\lambda) \oplus C$ are as well.  Therefore, we have the following $\tau$-invariant $G$-modules:
$$Q_r(\lambda) \oplus \left(\text{rad}_{G_r} \, C/\text{soc}_{G_r} \, C\right), \quad \left(\text{rad}_{G_r} \, Q_r(\lambda)/\text{soc}_{G_r} \, Q_r(\lambda) \right) \oplus C.$$
Let $V_1$ and $V_2$ denote these $G$-modules, with $M_1,N_1$ and $M_2,N_2$ designating these $G_r$-summands.

\begin{theorem}
Consider the modules $D_r(\lambda)$, $V_1$, and $V_2$ above, with the given $G_r$-decompositions.  If any of these modules has such a $G_r$-decomposition in which one of the $G_r$-summands is stable under the action of $U$, then there is a $G$-decomposition of the form
$$D_r(\lambda) \cong Q_r(\lambda) \oplus C.$$
In particular, the tilting module $T(2(p^r-1)\rho +w_0\lambda)$ is isomorphic to $Q_r(\lambda)$ over $G_r$.
\end{theorem}

\begin{proof}
The case for $D_r(\lambda)$ is a direct corollary to Theorem \ref{Utausplit}.  We will deal with the case of $V_1$, after which the $V_2$ case follows by a similar argument.

Suppose that some $G_r$-summand of $V_1$ isomorphic either to $M_1$ or $N_1$ is stable under the action of $U$.  Applying Theorem \ref{Utausplit}, there is a $G$-splitting of $V_1$ of this form.  This implies that there is a $G_r$-decomposition of $D_r(\lambda)$ into summands $Q_r(\lambda)$ and $C$ such that $Q_r(\lambda) \oplus \text{soc}_{G_r} \, C$ is a $G$-submodule of $D_r(\lambda)$.  Since
$$\text{Hom}_{G_r}(Q_r(\lambda), \text{soc}_{G_r} \, C)=0,$$
the $G_r$-summand of this submodule that is isomorphic to $Q_r(\lambda)$ is a $G$-submodule.  Thus, there is a $G$-submodule of $D_r(\lambda)$ that is isomorphic over $G_r$ to $Q_r(\lambda)$.  We are then done by Theorem \ref{Utausplit}.
\end{proof}

It would be nice if one could iterate the method in this theorem by further peeling off $G_r$-socle and $G_r$-radical layers from the summands in $V_1$ or $V_2$, and then continuing in this manner until a ``base case'' is reached at which point there is a splitting that can be brought all the way back to a splitting of the original module.  One would need additional information about the socle and radical filtrations of the projective indecomposable $G_r$-modules for this, and even then there appear to be other difficulties in extending the argument in this way.

\section{Reductive Groups Acting on $\mathbb{A}^n$}

Let $V$, $M$, and $N$ be as in Theorem \ref{maincase}.  As shown in the previous section, whether or not $V$ decomposes over $G$ into modules isomorphic over $G_r$ to $M$ and $N$ reduces to a question about fixed points for $G$ acting on some $\mathbb{A}^n$.  Unfortunately, determining when the latter holds may be an even harder issue to resolve.  In characteristic $0$, there are not any known examples of a reductive group acting on an affine space without a fixed point.  In characteristic $p$, however, Kraft and Popov \cite{KP} produce many such examples for any reductive group that is not a torus.  These examples arise from the fact that there always exist non-split extensions of the trivial module by some other module.

We now give a reformulation of their construction for $G$.  Consider finite dimenional $G$-modules $V$ and $W$ such that $W \le V$ and $V/W \cong k$.  The affine $G$-variety $\mathbb{P}(V)-\mathbb{P}(W)$ is isomorphic to $\mathbb{A}^m$, where $m = \text{dim}_k(W)$.

\begin{lemma}
For any $H \le G$, the variety $\mathbb{P}(V)-\mathbb{P}(W)$ has an $H$-fixed point if and only if the extension
$$0 \rightarrow W \rightarrow V \rightarrow k \rightarrow 0$$
splits over $H$.
\end{lemma}

\begin{proof}
A fixed point occurs on this space if and only if $\mathbb{P}(W)^H \subsetneq \mathbb{P}(V)^H$, which happens if and only if there is a one-dimensional $H$-submodule $V_0 \le V$ that is not contained in $W$.  This happens if and only if $V_0$ is an $H$-complement to $W$ in $V$, necessarily isomorphic to $k$ since $V/W \cong k$ over all subgroups of $G$.
\end{proof}

\begin{remark}
This was essentially the idea in \cite{KP}.  If $G$ is not a torus, then there is always a non-split extension of $k$ by some finite dimensional $G$-module, thus proving the existence of many fixed point free actions of $G$ on an affine space.
\end{remark}

In Theorem \ref{fixedislinear}, we showed that $\mathfrak{X}^V_M \times \mathfrak{X}^V_N$ has a fixed point for $H \le G$ if and only if it is linearizable for $H$.  The examples just given also have this property.

\begin{prop}
Let $H \le G$.  Suppose $H$ has a fixed point on $\mathbb{P}(V)-\mathbb{P}(W)$.  Then this action is linearizable for $H$.
\end{prop}

\begin{proof}
In this case we have the $H$-splitting $V = V_0 \oplus W$.  Thus we can choose a basis $\{v_0,\ldots, v_m\}$ for $V$ with $v_0$ spanning $V_0$ and $v_1,\ldots,v_m$ spanning $W$.  We can identify $\mathbb{P}(V)-\mathbb{P}(W)$ with the $G$-module $W$, according to the map
$$[a_0v_0, a_1v_1,\ldots, a_mv_m] \mapsto \frac{a_1}{a_0}v_1 + \cdots + \frac{a_m}{a_0}v_m.$$
The triviality of $H$ on $v_0$ evidently makes this an $H$-equivariant isomorphism.
\end{proof}

\begin{quest}
Suppose $\mathbb{A}^n$ has the structure of a $G$-variety in such a way that for every closed subgroup $H \le G$ there is an $H$-fixed point on $\mathbb{A}^n$ if and only if $\mathbb{A}^n$ is linearizable for $H$.  Does this $G$-structure arise as $\mathbb{P}(V)-\mathbb{P}(W)$ for some $W \le W$ with $V/W \cong k$?
\end{quest}

\section{Main Questions}

In addition to those raised earlier, there are two primary questions that we plan to address in future work.  The first is a general question dealing with the geometry of the $G$-varieties introduced in this paper.

\begin{quest}
Can the $G$-orbits in $\mathfrak{X}^V_M$ and $\mathfrak{Grass}_m(V)^{G_r}$ be described in some reasonable way?  What can be determined about the orbit closures?
\end{quest}

The hope, of course, is that there is some geometric information that can be extracted with limited knowledge of the $G$-submodule structure of $V$, and that this information will force the existence (or non-existence) of a $G$-fixed point in $\mathfrak{X}^V_M$.

The second question deals with the primary application we had in mind in initiating this work.

\begin{quest}
Can Donkin's tilting module conjecture be verified in more characteristics/ranks than those cases that are presently known? 
\end{quest}

We plan to address this question in a forthcoming paper, though so far the methods in that paper have been more inspired by, than dependent on, the results in this article.

\end{document}